\documentclass[11pt,leqno,english]{smfart}

\usepackage{dsfont,eucal,mathptmx,txfonts}
\usepackage[all]{xy}

\usepackage[text={6.5in,9.2in},centering]{geometry}

\newtheorem{theorem}{Theorem}[section] 

\newtheorem{lemma}[theorem]{Lemma}
\newtheorem{axiom}[theorem]{Axiom}
\newtheorem{corollary}[theorem]{Corollary} 
\theoremstyle{definition}

\usepackage[dvipsnames]{xcolor}   
\usepackage{xparse}
\usepackage{xr-hyper}
\definecolor{brightmaroon}{rgb}{0.76, 0.13, 0.28}
\usepackage[linktocpage=true,colorlinks=true,hyperindex,citecolor=teal,linkcolor=blue]{hyperref}   

\newcommand{\Ker}{\mathsf{Ker}}
\newcommand{\im}{\mathsf{Im}}

\begin{document}
\title{Isbell's subfactor projections in a noetherian form}

\author{Kishan Dayaram}
\address{Department of Mathematics and Applied Mathematics, University of Johannesburg, South Africa.}
\email{220157896@student.uj.ac.za}

\author{Amartya Goswami} 
\address{Department of Mathematics and Applied Mathematics, University of Johannesburg, South Africa; National Institute for Theoretical and Computational Sciences (NITheCS), Johannesburg, South Africa.}
\email{agoswami@uj.ac.za} 

\author{Zurab Janelidze}
\address{Mathematics Division, Department of Mathematical Sciences, Stellenbosch
University, Private Bag X1 Matieland, 7602, South Africa; National Institute for Theoretical and Computational Sciences (NITheCS), Stellenbosch, South Africa.}
\email{zurab@sun.ac.za}



\date{}

\maketitle

\begin{abstract} In this paper, we revisit the 1979 work of Isbell on subfactors of groups and their projections, which he uses to establish a stronger formulation of the butterfly lemma and its consequence, the refinement theorem for subnormal series of subgroups. We point out an error in the second part of Isbell's refinement theorem, but show that the rest of his results can be extended to the general self-dual context of a noetherian form, which includes in its scope all semi-abelian categories as well as all Grandis exact categories. Furthermore, we show that  Isbell's formulations of the butterfly lemma and the refinement theorem amount to canonicity of isomorphisms established in these results.
\end{abstract}

\section{Introduction}

Following Isbell \cite{Isb}, by a \emph{subfactor} of a group $G$ we mean a pair $(X,X^{+})$ of subgroups of $G$ such that $X \lhd X^{+}$ (i.e., $X$ is a normal subgroup of $X^{+}$). The \emph{projection} of a subfactor $(Y,Y^+)$ in a subfactor $(X,X^{+})$ is defined as the pair $((Y\cap X^+)X, (Y^+\cap  X^+)X)$. In \cite{Isb}, Isbell proves the following result on projections of subfactors. 

\begin{lemma}
    The projection of a subfactor $(Y,Y^+)$ in a subfactor $(X,X^+)$ is a subfactor $(M,M^+)$, and $(M,M^+)$ and $(X,X^+)$ have the same projection in $(Y,Y^+).$
\end{lemma}

\noindent As shown in \cite{Isb}, this lemma has an application in obtaining a stronger formulation of the butterfly lemma (also known as the Zassenhaus lemma) and the related (Schreier) refinement theorem for subnormal series of subgroups (see \cite{zassenhaus} and the references there for the more classical formulations of these results). One part of the claimed strength though is false as we show in this paper with a simple counterexample. The main goal of this paper is to show that the rest of Isbell's work extends to a very general `functorial' context: that of a noetherian form introduced and studied in \cite{dnaiv, Nie17, Nie19a, Nie19b, Day, Day1, Jan}. The core of such extension rests in finding a formal argument that could replace Isbell's element-based proof of the lemma above. It turns out that the auxiliary results obtained in \cite{dnaiv} are sufficient for this purpose: the crux of the argument rests on an easy adaptation of the `restricted modular law' established in \cite{dnaiv}. We also point out a link between the notion of a projection of a subfactor and an application of the universal isomorphism theorem from \cite{dnaiv} (Theorem~\ref{ThmA}), which in the case of the usual noetherian form for groups obtains the following formulation: given two subfactors $(X,X^+)$ and $(Y,Y^+)$ of a group $G$, the projection of $(X,X^+)$ in $(Y,Y^+)$ is $(Y,Y^+)$ (i.e., $(X,X^+)$ projects \emph{onto} $(Y,Y^+)$, in the terminology of Isbell) and vice versa, if and only if the relational composite of the canonical homomorphisms in the following zigzag is an isomorphism of groups:
$$\xymatrix{X^+/X & X^+\ar[l]\ar[r] & G & Y^+\ar[l]\ar[r] & Y^+/Y.}$$
This allows us to get a more direct proof of an isomorphism in the butterfly lemma than the one contained in \cite{Isb}; it also allows us to see that Isbell's stronger formulations of the butterfly lemma and the refinement theorem essentially reduce to canonicity of isomorphisms established in these results. 

Knowing that Isbell's results can be extended to an arbitrary noetherian form means that they are applicable to all semi-abelian categories \cite{semi abelian} (and their duals), as well as all Grandis exact categories \cite{Gra}, which in the pointed case become categories studied in \cite{Pup62,Mit65}; and recall that the overlap of the classes of semi-abelian and Grandis exact categories is the class of abelian categories. In particular, the following categories have noetherian forms. 
\begin{itemize}
\item Categories of the following group-like structures (because they are semi-abelian): modules over a ring, (non-abelian) groups, non-unitary rings, algebras over rings, Lie algebras, Boolean algebras, Heyting semilattices, crossed modules, compact Hausdorff groups, cocommutative Hopf algebras over a field, and categories of many other group-like structures.

\item The categories of distributive and modular lattices with modular connections as morphisms (a modular connection is a Galois connection satisfying the formulas given in Axiom~\ref{AxA} in this paper), because these categories are Grandis exact.

\item The category of sets with partial functions as morphisms, which is equivalent to the category of pointed sets, because the duals of these categories are semi-abelian. More generally, the categories of pointed objects in toposes (whose duals are also semi-abelian). 
\end{itemize}
Noetherian forms also include categories that are not necessarily (co-)semi-abelian or Grandis exact. According to \cite{Nie}, for instance, they include all algebraic categories. Every topos also possesses a noetherian form, as shown in \cite{Jan}. Another example is given by the category of unitary rings. In many of these examples, however, the refinement theorem trivializes since a subnormal series there is always of length at most~$1$.

A noetherian form over a category $\mathbb{C}$ is a faithful amnestic functor $F$ whose codomain is $\mathbb{C}$, satisfying the axioms presented in \cite{dnaiv}. The fibres of this functor are viewed as (abstract) subobject posets. We follow the convention in \cite{dnaiv} and call objects of $\mathbb{C}$ \emph{groups}, morphisms in $\mathbb{C}$ \emph{homomorphisms}, and objects in the fibre of an object (group) $G$, the \emph{subgroups} of $G$. This hints to how the noetherian form context specializes to the usual group theory. In particular, the functor $F$ in this case is the bifibration of subgroups $F\colon \mathbf{Grp}_2\to \mathbf{Grp}$. As noted above, noetherian forms also have numerous other examples (see also \cite{dnaiv, Nie, Jan}); the reason for having terminology akin to the context of group theory is that this context is the motivating example: as in the present paper, the results in the context of a noetherian form are usually obtained by abstracting theorems of group theory (the use of the group-theoretic language makes it easier to compare the abstracted results with their original counterparts). 

In this paper we rely on the groundwork material pertaining to noetherian forms as laid down in \cite{dnaiv}. For reader's convenience, however, we nevertheless include below, in Section~\ref{SecA}, a brief refresher on basic terminology, notation, axioms and results in the context of a noetherian form. We include in the same section some new observations that can be easily deduced from the results contained in \cite{dnaiv} and which help with the rest of the paper. In Section~\ref{SecB}, we extend Isbell's work from \cite{Isb} to the context of a noetherian form. In Section~\ref{SecC}, we provide a counterexample to one of the claims in the refinement theorem formulated in \cite{Isb}. In Section~\ref{SecD}, we briefly discuss some similarities and differences with other abstract approaches to the butterfly lemma and the refinement theorem in the literature.

\section{Preliminaries}\label{SecA}

By a \emph{form} we mean a faithful amnestic functor, as in \cite{dnai} (see also \cite{Jan14}). Recall that for a functor $F$ and an object $G$ in the codomain of $F$, the fibre $F^{-1}(G)$ is the subcategory of the domain of $F$ consisting of all morphisms that map by $F$ to $\mathsf{id}_G$, the identity morphism of $G$. When $F$ is faithful, each $F^{-1}(G)$ is a preorder; in this case, for $F$ to be  amnestic means that the preorder relation is antisymmetric.

A \emph{noetherian form} is a form satisfying the axioms stated below. With the terminology recalled in the Introduction, in the context of a noetherian form $F$, the subgroups of a group $G$ form a poset --- the fibre $F^{-1}(G)$, which we write as $\mathsf{Sub} (G)$. For any homomorphism $f\colon G\to H$, we write $X\subseteq_f Y$ where $X\in \mathsf{Sub} (G)$ and $Y\in \mathsf{Sub} (H)$, when there exists a morphism $f'\colon X\to Y$ such that $F(f')=f$. In the case when $f=\mathsf{id}_G$, this is the partial order relation of $\mathsf{Sub} (G)$, which we denote simply by `$\subseteq$' (drop the subscript). Note that the transitivity law of this partial order is a special case of a more general property, which arises from the fact that $F$ preserves composition:
$$W\subseteq_{g} X\subseteq_{f}Y\quad\Rightarrow\quad W\subseteq_{fg} Y.$$

In the context of the usual group theory, $F$ is the bifibration of subgroups. Write $\mathbf{Grp}_2$ for the category of pairs $(G,S)$ where $G$ is a group and $S$ is a subgroup of $G$. A morphism $f\colon (G,X)\to (H,Y)$ is a group homomorphism $f\colon G\to H$ such that for every $x\in X$ we have $f(x)\in Y$. Then $F$ is the codomain functor $(G,X)\mapsto G$. Note that the relation $X\subseteq_f Y$ is then the relation `for every $x\in X$ we have $f(x)\in Y$'. It can be equivalently reformulated using the direct image map as $f(X)\subseteq Y$, or using the inverse image map as $X\subseteq f^{-1}(Y)$. In the following axiom the direct image and the inverse image mappings are introduced abstractly. The axiom below is in fact equivalent to requiring that the form $F$ is a Grothendieck bifibration (an obvious fact for anyone familiar with the latter notion).

\begin{axiom}\label{AxD}
    For each homomorphism $f\colon G \rightarrow H$ there are maps \begin{center}
	$\xymatrix {\mathsf{Sub} (G) \ar @/^/[r]^f  &\mathsf{Sub} (H) \ar@/^/[l]^{f^{-1}} }$
\end{center}
such that the equivalences
$$f(X)\subseteq Y\quad\Leftrightarrow\quad X\subseteq_f Y\quad\Leftrightarrow\quad X\subseteq f^{-1}(Y)$$
hold for all $X\in\mathsf{Sub} (G)$ and $Y\in\mathsf{Sub} (H)$.
Furthermore, for two composable homomorphisms $f\colon G \to H$ and $g\colon E \to G$, and subgroups $W$ of $E$ and $Y$ of $H$, we have $(fg)(W)=f(g(W))$ (and $(fg)^{-1} (Y)=g^{-1}(f^{-1}(Y))$, which can be derived). Lastly, for any group $G$ and subgroup $X$ of $G$, we have $\mathsf{id}_G(X)=X$ (and $\mathsf{id}_G^{-1}(X)=X$, which can be derived).
\end{axiom}

The axiom above implies that the two maps $f(-),f^{-1}(-)$ form a monotone Galois connection (i.e., an adjunction) between $\mathsf{Sub} (G)$ and $\mathsf{Sub} (H)$, where the \emph{direct image map} $f(-)$ is the left adjoint and the \emph{inverse image map} $f^{-1}(-)$ is a right adjoint. Furthermore, it is easy to check that this Galois connection is in fact uniquely determined by $f$, thanks to the equivalences required in the axiom above. The usual and well-known properties of Galois connections (or more generally, of adjunctions) are sufficient to derive the conclusions claimed in brackets in the axiom above. The same usual properties yield that the direct image maps preserve joins (when they exist) and the inverse image maps preserve meets (when they exist).

For a group $G$, we denote the bottom and top elements of $\mathsf{Sub} (G)$, if they exist, by $0_G$ and $1_G$ respectively (the empty join and the empty meet), where the subscripts are dropped when it is clear which group we are referring to. Consider a homomorphism $f\colon G\to H$. The \emph{kernel} of $f$, denoted $\Ker f$, is defined as $f^{-1}(0_H)$ (the inverse image of the smallest subgroup of $H$). Similarly, the \emph{image} of $f$, denoted $\im f$, is defined as $f(1_G)$ (the direct image of the largest subgroup). These terms obtain their usual meaning in usual group theory. The axiom above guarantees that similarly as in usual group theory, the kernel of an isomorphism is the least subgroup and the image is the largest subgroup.

We say that a subgroup is \emph{conormal} if it is the image of some homomorphism and we say that a subgroup is \emph{normal} if it is the kernel of some homomorphism. Note that in the concrete case of the usual groups, every subgroup is conormal. We do not want to require this in the abstract case, since we want the axioms of a noetherian form to be self-dual relative to the `functorial duality' of the intrinsic language of a noetherian form: the duality arising by replacing a noetherian form $F\colon\mathbb{X}\to \mathbb{C}$ with the dual functor $F^\mathsf{op}\colon \mathbb{X}^\mathsf{op}\to \mathbb{C}^\mathsf{op}$.
See Table~\ref{FigA} for the resulting duality scheme: it shows the effect of duality on primitive expressions in the internal language of a form.

\begin{table}        
\begin{tabular}{|c|c|}
\hline 
\textbf{Expression} & \textbf{Dual Expression} \\
\hline
$G\in\mathbb{C}$ & $G\in\mathbb{C}$\\
\hline
$X\in F^{-1}(G)$ & $X\in F^{-1}(G)$ \\
\hline
$f\colon G\rightarrow H$ in $\mathbb{C}$ &$f\colon H\rightarrow G$ in $\mathbb{C}$ \\
\hline
$X\subseteq_f Y$ & $Y\subseteq_f X$ \\
\hline
$gf$ & $fg$\\
\hline 
\end{tabular}
\centering
\caption{Scheme for functorial duality of the internal language of a form $F\colon\mathbb{X}\to\mathbb{C}$}
\label{FigA}
\end{table}

\begin{axiom}\label{AxA} For any group $G$, the subgroup poset $\mathsf{Sub} (G)$ is a bounded lattice. Moreover, for any homomorphism $f\colon G \rightarrow H$ and subgroups $X$ of $G$ and $Y$ of $H$, we have $$f^{-1}f(X)=X\vee \Ker f \text{ and } ff^{-1}(Y)=Y\wedge \im f.$$
\end{axiom} 

The second half of the axiom above can be reformulated as an abstraction of the classical lattice isomorphism theorem stated below, where by an \emph{interval} $[X,Y]$ of subgroups of a group $G$ we mean the subposet of $\mathsf{Sub}(G)$ given by all $Z\in \mathsf{Sub}(G)$ such that $X\subseteq Z\subseteq Y$. As a convention, when we say that $[X,Y]$ is an interval, we assume that $X\subseteq Y$.

\begin{theorem}[Lattice Isomorphism Theorem]\label{LIT}
For any homomorphism $f\colon G\to H$, the direct and inverse image maps along $f$ result in an isomorphism of the intervals $[\mathsf{Ker}f,1_G]$ and $[0_H,\mathsf{Im}f]$, which are sublattices of $\mathsf{Sub}(G)$ and $\mathsf{Sub}(H)$ (i.e., they are closed under binary meet and join), respectively. 
\end{theorem}

See \cite{Jan14} for various other equivalent reformulations of the axiom above and links with closely related axioms considered in universal and categorical algebra. The simplest reformulation, which only makes use of the poset structure of lattices, is given by requiring that the set $\{f(X)\mid X\in\mathsf{Sub}(G)\}$ is down-closed and the set $\{f^{-1}(Y)\mid Y\in\mathsf{Sub}(H)\}$ is up-closed. In \cite{Jan14}, following the terminology of Grandis \cite{Gra, Gra2}, we may call monotone maps $f$ between posets that admit a right adjoint $f^{-1}$ such that the two maps have these properties \emph{modular} maps. A useful fact about modular maps is that their composite is also modular. In \cite{Gra}, a Galois connection between modular lattices is called a modular connection when its left adjoint is modular in our sense.   

The second of the last two equalities in Axiom~\ref{AxA} is a weaker form of what, in the context of `hyperdoctrines', Lawvere coined \emph{Frobenius reciprocity} in \cite{Law70}:
\begin{equation}\label{EquB}f(f^{-1}(Y)\wedge X)=Y\wedge f(X).\end{equation}
To get the former equality, simply put $X=1$ in the latter equality. Note that the dual of this stronger property,
$$f^{-1}(f(X)\vee Y)=X\vee f^{-1}(Y),$$
fails in group theory as it implies modularity of the subgroup lattices. Indeed, by setting $f$ to be an inclusion of a subgroup $Z$ containing $X$, in which case $f(X)=X$ and $f^{-1}(H)=H\wedge Z$, the equality above becomes precisely the modular law:
\begin{equation}\label{EquA}
X\subseteq Z\quad\Rightarrow\quad (X\vee Y)\wedge Z=X\vee(Y\wedge Z).
\end{equation}
As it is well known, the dihedral group of order four is the smallest group whose subgroup lattice is not modular. 

The axioms above do imply the dual of Frobenius reciprocity in the case when $Y$ is normal, i.e., $Y=\mathsf{Ker}g$ for some homomorphism $g$:
$$f^{-1}(f(X)\vee \mathsf{Ker}g)=f^{-1}g^{-1}gf(X)=(gf)^{-1}(gf)(X)=X\vee \mathsf{Ker}gf=X\vee f^{-1}(\mathsf{Ker}(g)).$$
Note that the case when $Y=0$ is a special case of this, since $0$ is normal. So the second half of the axiom above can be equivalently reformulated to require that Frobenius reciprocity holds when $X$ is conormal together with the dual of this property.

The axioms above are sufficient to establish the following lemma in our abstract context, where the two halves of the lemma are acutally dual to each other.

\begin{theorem}[Restricted Modular Law, \cite{dnaiv}]\label{RML}
 For any three subgroups, $X$, $Y$, and $Z$ of a group $G$, if $Y$ is normal and $Z$ is conormal, then the implication (\ref{EquA}) holds. The same implication holds also when $Y$ is conormal and $X$ is normal.
\end{theorem}

In the case of a concrete group $G$ and subgroups $X$ and $Y$, when either $X$ or $Y$ is normal, the join of $X$ and $Y$ is given by
$$X\vee Y=XY=\{xy\mid [x\in X ]\& [y\in Y]\}.$$
The implication above, where $X$ is normal, then becomes
\begin{equation}
\label{EquC}        X\subseteq Z \quad\Rightarrow\quad X(Y\cap Z)=XY\cap Z.
\end{equation}
This implication holds even when $X$ and $Y$ are merely subsets of $G$. So, in particular, it holds when $X$ and $Y$ are subgroups of $G$ with $X$ not necessarily being a normal subgroup of $G$. In our abstract context, we cannot speak about subsets of groups unless they are subgroups, and neither can we consider the product of two subsets; the second half of the restricted modular law is then the closest we can get to (\ref{EquC}). 

Consider a subgroup $X$ of a group $G$. An \emph{embedding} associated to $X$ is a homomorphism $\iota_X \colon X \to G$, which may or may not exist, defined by a universal property. Note that here the group $X$, i.e., the domain of $\iota_X$, is not the same thing as the subgroup $X$ (they are in fact objects in different categories), even though we write them the same way. In usual group theory, both $X$'s will be the same group. The universal property states: $\im \iota_X \subseteq X$ and for any homomorphism $f\colon H \to G$ with $\im f \subseteq X$, there exists a unique homomorphism $f' \colon H \to X $ which satisfies $f=\iota_X f'$. It is easy to see that the identity homomorphism $\mathsf{id}_G\colon G\to G$ is an embedding associated to the largest subgroup $1_G$ of a group $G$. It is then in line with the convention above to write $G$ for the largest subgroup $1_G$ of $G$. 

Dually, a \emph{projection} associated to subgroup $X$ of a group $G$ is a homomorphism $\pi_X \colon G \to G/X$ such that $X \subseteq \Ker \pi_X $ and for any homomorphism $h\colon G \to E$ with $X\subseteq \Ker h$, there exists a unique homomorphism $h' \colon G/X \to E$ which satisfies $h=h' \pi_X$. We will refer to the projections as \emph{projection homomorphisms} since the term `projection' will also have a second meaning in this paper (see the next section). We will use arrows of the form $$\xymatrix{\ar@{^(->}[r]& &\text{and}& \ar@{->>}[r]&}$$ to denote embeddings and projection homomorphisms respectively.

In usual group theory, an embedding associated to a subgroup is given by its realization as a group equipped with the subgroup inclusion homomorphism, while a projection homomorphism associated to a subgroup is given by the quotient map corresponding to the smallest normal subgroup containing the given one.

\begin{axiom}\label{AxF} Any homomorphism $f\colon X \rightarrow Y$ factorizes as $f=\iota_{\im f} h \pi_{\Ker f}$ where $h$ is an isomorphism (this statement subsumes the statement of existence of $\iota_{\im f}$ and $\pi_{\Ker f}$). 
\end{axiom}

Note that the axiom above implies that for any conormal subgroup, an embedding associated to it exists, and dually, for any normal subgroup, a projection associated to it exists.

\begin{axiom}\label{AxB}
The join of two normal subgroups is normal and the meet of two conormal subgroups is conormal.
\end{axiom}

In the usual group theory, projection homomorphisms are the same as surjective homomorphisms of groups, while embeddings are the same as injective group homomorphisms. Here, as well as in our abstract context, these classes of morphisms contain isomorphisms and are closed under composition. In fact, they form a proper factorization system. Moreover, as we know from \cite{dnaiv}, projection homomorphisms are precisely those homomorphisms $f$ such that the direct image map $f(-)$ is surjective. By general properties of Galois connections, surjectivity of $f(-)$ is equivalent to injectivity of the inverse image map $f^{-1}(-)$ and is further equivalent to $f(-)$ being a left inverse of $f^{-1}(-)$. Dually, embeddings are precisely those homomorphisms $f$ such that the direct image map $f(-)$ is injective, or equivalently, the inverse image map $f^{-1}(-)$ is surjective, or yet equivalently, $f(-)$ is a right inverse of $f^{-1}(-)$. Another way to characterize embeddings and projection homomorphisms is as follows: embeddings are precisely those homomorphisms $f\colon G\to H$ for which $\mathsf{Ker}f=0_G$; dually, projection homomorphisms are precisely those $f\colon G\to H$ for which $\mathsf{Im}f=H$.

Another useful property of embeddings and projection homomorphisms that was noted in \cite{dnaiv} is that any embedding is an embedding associated to its image, and moreover, if an embedding is associated with a conormal subgroup, then its image equals that conormal subgroup. Dually, any projection homomorphism is a projection homomorphism associated with its kernel and if it is associated with a normal subgroup, then the kernel equals that normal subgroup.

As shown in \cite{dnaiv}, the last axiom is actually equivalent to requiring that the inverse image of a conormal subgroup along an embedding is conormal and the direct image of a normal subgroup along a projection homomorphism is normal.

By a \emph{zigzag} we mean a diagram $$\xymatrix{G_0 \ar@{-}[r]^-{f_1} & G_1 \ar@{-}[r]^-{f_2}&G_2\ar@{-}[r]^-{f_3} &G_3 \ar@{-}[r]^-{f_4}&G_4\ar@{.}[r]&G_{n-1}\ar@{-}[r]^-{f_n}&G_n }$$
of groups and homomorphisms having possibly alternating directions (hence the omission of the arrowheads in the diagram). We say that a subgroup $Y$ of $G_n$ is obtained by \emph{chasing} a subgroup $X$ of $G_0$ forward along the zigzag if $$Y=f_1^* f_2^* f_3^* \dots f_n^*(X)$$ where $f_i^* =f_i$ if the arrowhead of $f_i $ appears on the right and $f_i^*=f_i^{-1} $ if the arrowhead of $X_i$ appears on the left. We say that a subgroup $X$ of $G_0$ is obtained by chasing a subgroup $Y$ of $G_n$ backward along the zigzag if $X$ if obtained by chasing $Y$ forward along the zigzag:
$$\xymatrix{G_n \ar@{-}[r]^{f_n} & G_{n-1} \ar@{.}[r]&G_4\ar@{-}[r]^{f_4} &G_3 \ar@{-}[r]^{f_3}&G_2\ar@{-}[r]^{f_2}&G_1\ar@{-}[r]^{f_1}&G_0. }$$ In usual group theory, we can view each homomorphism in the zigzag as a relation directed from left to right: if $f_i$ is directed right, then it is $f_i$ regarded as a relation, while $f_i$ is directed left, then we can consider the opposite relation of $f_i$. Then, it may or may not be that the composite of these relations is a homomorphism. When it is, we refer to this homomorphism as one `induced' by the zigzag. In \cite{dnaiv}, the concept of an induced homomorphism has been abstracted to the context of an arbitrary noetherian form and the following result has been established.

\begin{theorem}[Homomorphism Induction Theorem, \cite{dnaiv}] \label{HIT}
For a zigzag to induce a homomorphism it is necessary and sufficient that chasing the least subgroup of the initial node forward along the zigzag results in the least subgroup of the final node, and chasing the largest subgroup of the final node backward along the zigzag results in the largest subgroup of the initial node. Moreover, when a zigzag induces a homomorphism, the induced homomorphism is unique and the direct and inverse image maps of the induced homomorphism are given by chasing a subgroup forward and backward, respectively, along the zigzag.
\end{theorem}

We can chase intervals along zigzags by chasing its least and largest elements. The Universal Isomorphism Theorem from \cite{dnaiv} can then be reformulated in the following way (as in \cite{dnaiv}, it is a corollary of the Homomorphism Induction Theorem).

\begin{corollary}[Universal Isomorphism Theorem]\label{UIT}
The interval $[0,1]$ in each of the two ends of the zigzag chases to the interval $[0,1]$ in the other end if and only if the zigzag induces an isomorphism.
\end{corollary}

Given subgroups $X$ and $Y$ of a group $G$, we say that $X$ is \textit{normal} to $Y$ and we write $X \lhd Y$ if $X\subseteq Y$, $Y$ is conormal and $\iota_Y^{-1}(X)$ is normal. In usual group theory, the relative normality relation is of course the usual relation of one subgroup being normal inside another. 
Every pair $(X,Y)$ such that $X\lhd Y$ gives rise to a zigzag,
\begin{equation}
\label{EquH}
\xymatrix@!=40pt{G &Y \ar@{_(->}[l]_-{\iota_{Y}} \ar@{->>}[r]^-{\pi_{\iota_{Y}^{-1}(X)}} &Y/X.}\end{equation}
The following lemma will be useful.

\begin{lemma}\label{LemA}
For two subgroups $X \lhd Y$ of a group $G$, chasing the interval $[0,1]=[0_{Y/X},Y/X]$ in $\mathsf{Sub}(Y/X)$ backward along the zigzag (\ref{EquH}) results in the interval $[X,Y]$. Furthermore, chasing subgroups backward along the zigzag (\ref{EquH}) gives an isomorphism between the lattice $\mathsf{Sub}(Y/X)$ and the sublattice $[X,Y]$ of $G$, whose inverse is given by chasing a subgroup $Z\in[X,Y]$ forward along the same zigzag. Finally, for any subgroup $Z$ of $G$, chasing $Z$ forward and then backward along the zigzag (\ref{EquH}) results in the subgroup $(Z\wedge Y)\vee X$.      
\end{lemma}

\begin{proof}
The effect of chasing $Z$ forward and then backward along the zigzag (\ref{EquH}) is:
$$\iota_Y\pi_{\iota_{Y}^{-1}(X)}^{-1}\pi_{\iota_{Y}^{-1}(X)}\iota_Y^{-1}(Z)=\iota_Y(\iota_Y^{-1}(Z)\vee \iota_Y^{-1}(X))=\iota_Y\iota_Y^{-1}(Z)\vee\iota_Y\iota_Y^{-1}(X)=(Z\wedge Y)\vee X.$$
Since $\iota_Y$ is an embedding and $\pi_{\iota^{-1}_Y(X)}$ is a projection, chasing a subgroup backward along the zigzag (\ref{EquH}) and then forward gives back the same subgroup. So we get an isomorphism between $\mathsf{Sub}(Y/X)$ and the subposet of $\mathsf{Sub}(G)$ given by those $Z\in \mathsf{Sub}(G)$ such that $Z=(Z\wedge Y)\vee X$. It is easy to see that the sublattice $[X,Y]$ of $\mathsf{Sub}(G)$ is in fact such subposet of $\mathsf{Sub}(G)$.   
\end{proof}

Isbell's `subfactors' \cite{Isb} are precisely the pairs $(X,X^+)$ such that $X\lhd X^+$, while the construction $(Z\wedge X^+)\vee X$ is what he calls the `projection' of $Z$ in the subfactor $(X,X^+)$ (not to be confused with our `projection homomorphism'). In light of the lemma above, we can view the projection of $Z$ in a subfactor $(X,X^+)$ as a canonical way of transporting $Z$ to the interval $[X,X^+]$. We will write Isbell's subfactors as intervals $X=[X^-,X^+]$.

The following two lemmas from \cite{dnaiv} abstract important properties of relative normality from usual group theory. They will be used, among other results discussed above, in generalizing Isbell's analysis of subfactors to the context of a noetherian form. 

\begin{lemma}[\cite{dnaiv}]\label{meet nor}
    Let $W,X$ and $Y$ be subgroups of a group $G$ such that $W\lhd X$ and $Y$ is conormal. Then $W\wedge Y \lhd X\wedge Y$. In particular, for $W\subseteq Y \subseteq X$, we obtain $W\lhd Y$.
\end{lemma}
\begin{lemma}[\cite{dnaiv}]\label{join nor}
   Let $W,X$ and $Y$ be subgroups of a group $G$ such that $W\lhd X$ and $Y \lhd X \vee Y$. Then $W\vee Y \lhd X\vee Y$. In particular, for  $Y\subseteq X$, the assumptions become $W\lhd X$ and $Y\lhd X$ and they yield $W\vee Y \lhd X.$
\end{lemma}

The Restricted Modular Law (Theorem~\ref{RML}) has the following consequence.

\begin{corollary}[Less Restricted Modular Law]\label{RML2}
For any three subgroups, $X$, $Y$, and $Z$ of a group $G$, if for some subgroup $S$ of $G$ we have $Y\lhd S$, $Z\subseteq S$ and $Z$ is conormal, then the implication (\ref{EquA}) holds. The same implication holds when $Y$ is conormal, $Y\subseteq S$, $Z\subseteq S$ and $X\lhd S$ for some subgroup $S$ of $G$.
\end{corollary}

\begin{proof}
Suppose $X\subseteq Z$. Under each set of assumptions, we have $X,Y,Z\subseteq S$. We then get:
\begin{align*}
\iota_S\iota_S^{-1}((X\vee Y)\wedge Z) &=\iota_S(\iota_S^{-1}(X\vee Y)\wedge \iota_S^{-1}(Z))\\
&=\iota_S((\iota_S^{-1}(X)\vee \iota_S^{-1}(Y))\wedge \iota_S^{-1}(Z)) & \textrm{(Theorem~\ref{LIT})}\\
&=\iota_S(\iota_S^{-1}(X)\vee (\iota_S^{-1}(Y)\wedge \iota_S^{-1}(Z))) & \textrm{(Theorem~\ref{RML})}\\
&=\iota_S(\iota_S^{-1}(X)\vee \iota_S^{-1}(Y\wedge Z))\\
&=\iota_S(\iota_S^{-1}(X))\vee \iota_S(\iota_S^{-1}(Y\wedge Z))\\
&=X\vee (Y\wedge Z). & \qedhere
\end{align*}
\end{proof}
 
\section{Isbell's results on subfactors}\label{SecB}

In what follows we work in $\mathsf{Sub} (G)$, for a fixed $G$. An interval $X=[X^-,X^+]$ is a \emph{subfactor} (of $G$) when $X^-\lhd X^+$. The \emph{projection} of a subgroup $Z$ in an interval $X=[X^-,X^+]$ is defined as
$$ZX=(Z\wedge X^+)\vee X^-.$$
Notice that $ZX\in X$. Projection is monotone (in the first argument): if $Z_1\subseteq Z_2$ then $Z_1X\subseteq Z_2X$.
This yields that the projection of an interval $Y=[Y^-,Y^+]$ in an interval $X=[X^-,X^+]$,
$$YX=[Y^-X,Y^+X]=[(Y^-\wedge X^+)\vee X^-,(Y^+\wedge X^+)\vee X^-],$$
is again an interval (i.e., $Y^-X\subseteq Y^+X$).  
Notice that $YX\subseteq X$. 
We say that $Y$ projects \emph{onto} $X$ when $YX=X$.

We are now ready to abstract Lemma~1 from \cite{Isb}, of which the butterfly lemma in \cite{Isb} is a corollary.

\begin{lemma}[Subfactor Projection Lemma]\label{proj} Given two intervals $X$ and $Y$, we have
$$(YX)Y=XY$$
provided $X$ is a subfactor. Moreover, $YX$ is a subfactor provided $X$ is a subfactor and $(YX)^+=Y^+X$ is conormal.
\end{lemma}

\begin{proof}
We have $((YX)Y)^+=(YX)^+Y=(Y^+X)Y\subseteq X^+Y=(XY)^+$ by monotonicity of projection. At the same time, 
\begin{align*} ((YX)Y)^+ &=(((Y^+\wedge X^+)\vee X^-)\wedge Y^+)\vee Y^- \\ &\supseteq (Y^+\wedge X^+ \wedge Y^+) \vee Y^-\\ &=(X^+\wedge Y^+)\vee Y^-\\ &=(XY)^+
\end{align*} and so $((YX)Y)^+=(XY)^+$. Furthermore, assuming that $X$ is a subfactor, we have:
\begin{align*}
        ((YX)Y)^- &=(((Y^-\wedge X^+)\vee X^-)\wedge Y^+)\vee Y^-\\
        &=((((Y^-\wedge X^+)\vee X^-)\wedge Y^+) \wedge X^+)\vee Y^-\\
        &=(((Y^-\wedge X^+)\vee X^-)\wedge (Y^+ \wedge X^+))\vee Y^-\\
        &= (Y^-\wedge X^+) \vee (X^-\wedge (Y^+\wedge X^+))\vee Y^- & (\textrm{Corollary}~\ref{RML2}) \\
        &=(Y^-\wedge X^+)\vee (X^-\wedge Y^+)\vee Y^-\\
        &=(X^-\wedge Y^+)\vee Y^-\\
        &=(XY)^-
\end{align*}
    
Now, suppose $X$ is a subfactor and $(YX)^+$ is conormal. By Lemma~\ref{meet nor}, we have $X^-\lhd (YX)^+$ since $(YX)^+$ is conormal, $X^-\subseteq (YX)^+\subseteq X^+$, and $X^-\lhd X^+$. We also have $Y^- \wedge X^+\lhd Y^+ \wedge X^+$ by the same lemma. So by applying Lemma~\ref{join nor}, we obtain $(YX)^-=(Y^-\wedge X^+)\vee X^- \lhd (Y^+ \wedge X^+)\vee X^-=(YX)^+.$
\end{proof}

The following result relates Isbell's concept of projection onto and homomorphism induction in the sense of \cite{dnaiv}. Not only does this result help to obtain an isomorphism of the butterfly lemma quicker than it is done in \cite{Isb}, but it also emphasizes canonicity of such isomorphism.

\begin{theorem}\label{ThmA}
Given subfactors $X$ and $Y$ of a group $G$, consider the zigzag
$$\xymatrix@!=60pt{ X^+/X^- & X^+\ar@{->>}[l]_-{\pi_{\iota_{X^+}^{-1}(X^-)}}\ar@{^(->}[r]^-{\iota_{X^+}} & G & Y^+ \ar@{_(->}[l]_-{\iota_{Y^+}} \ar@{->>}[r]^-{\pi_{\iota_{Y^+}^{-1}(Y^-)}} & Y^+/Y^-.}$$
This zigzag induces an isomorphism if and only if $XY=Y$ and $YX=X$.
\end{theorem}

\begin{proof}
By Lemma~\ref{LemA}, $XY=Y$ is equivalent to the interval $[0,1]$ at $X^+/X^-$ chasing to the interval $[0,1]$ at $Y^+/Y^-$. Symmetrically, $YX=X$ is equivalent the interval $[0,1]$ at $Y^+/Y^-$ chasing to the interval $[0,1]$ at $X^+/X^-$. So the result follows from the Universal Isomorphism Theorem.
\end{proof}

\begin{lemma}[Butterfly Lemma] \label{butterfly}
For any two subfactors $X$ and $Y$ of a group $G$ we have: $$(YX)(XY)=XY,\quad (XY)(YX)=YX.$$ When $Y^+X$ and $X^+Y$ are conormal, $XY$ and $YX$ are also subfactors and  $Y^+X/Y^-X$ is isomorphic to $X^+Y/X^-Y$.
\end{lemma}
\begin{proof}
From Lemma~\ref{proj} it follows that \begin{align*}
(YX)^-\wedge(XY)^+ &\subseteq (YX)^-\wedge Y^+\\ &\subseteq((YX)^-\wedge Y^+)\vee Y^- \\ &=(YX)^-Y\\ &=((YX)Y)^-\\ &=(XY)^-.
\end{align*}
Therefore, $(YX)^-(XY)=(XY)^-$. We also have $(YX)^+(XY)=(XY)^+$ since, similarly as above, 
\begin{align*}(YX)^+\wedge(XY)^+ &\subseteq (YX)^+\wedge Y^+\\ &\subseteq((YX)^+\wedge Y^+)\vee Y^-=(YX)^+Y\\ &=((YX)Y)^+\\ &=(XY)^+.\end{align*}
So $(YX)(XY)=XY$. Symmetrically, we get $(XY)(YX)=YX$. By Lemma~\ref{proj}, when $Y^+X$ and $X^+Y$ are conormal, $XY$ and $YX$ are subfactors. Since they project onto each other by the first part of the lemma, the corresponding quotients $$Y^+X/Y^-X=(YX)^+/(YX)^-\textrm{ and }X^+Y/X^- Y=(XY)^+/(XY)^-$$ must then be isomorphic by Theorem~\ref{ThmA}. 
\end{proof}

In the abstract context of a noetherian form, we define the classical notion of a \emph{subnormal series} (also called \emph{normal series}) of a group $G$ to be a chain of distinct subgroups 
\begin{equation}\label{EquD}
G=X_0 \supset X_1 \supset \dots \supset X_n=0_G
\end{equation}
such that each interval $[X_{i+1},X_i]$ is a subfactor. Here $A\supset B$ is the conjunction of $B\subseteq A$ and $A\neq B$. A \textit{refinement} of a subnormal series is a subnormal series that contains all the subgroups of the original series. Following Isbell, we further define a subnormal series (\ref{EquD}) of a group $G$ to be \textit{projectively isomorphic} to a subnormal series \begin{equation}\label{EquE} G=Y_0 \supset Y_1 \supset \dots \supset Y_m=0_G
\end{equation} of $G$ if each subfactor $[X_{i+1},X_i]$ of the first series projects onto a subfactor $[Y_{j+1},Y_j]$ of the second series such that $[Y_{j+1},Y_j]$ projects onto $[X_{i+1},X_i]$. Note that as defined, the relation of being `projectively isomorphic' is not symmetric. Symmetry of this relation follows from the second half of the following lemma, which abstracts Lemma~2 from \cite{Isb}.

\begin{lemma}\label{proj iso terms}
Given a subnormal series (\ref{EquE}), if $[Y_{j+1},Y_j]$ projects onto a subfactor $X=[X^-,X^+]$, then $X^-$ and $X^+$ have the same projection in any other subfactor $[Y_{k+1},Y_k]$ where $k\neq j$. In particular, if a subnormal series (\ref{EquD}) is projectively isomorphic to a subnormal series (\ref{EquE}), then the two subnormal series have the same number of terms and each subfactor $[X_{i+1},X_i]$ in the first series projects onto a unique subfactor $[Y_{j+1},Y_{j}]$ in the other series (and $[Y_{j+1},Y_{j}]$ projects onto $[X_{i+1},X_i]$).  
\end{lemma}

\begin{proof}
Suppose $[Y_{j+1},Y_j]$ projects onto a subfactor $X=[X^-,X^+]$. We will show that $X^-$ and $X^+$ have the same projection in every other subfactor $[Y_{k+1},Y_k]$. First, we consider the case when $k<j$. In this case, $Y_j\subseteq Y_{k+1}\subseteq Y_k$. Then, since the projection of $Y_j$ in $X$ is $X^+$, the projection of $[Y_{k+1},Y_k]$ in $X$ is $[X^+,X^+]$. After this, it follows by Lemma~\ref{proj} that $X$ and $X^+$ have the same projection in $[Y_{k+1},Y_k]$. In the case when $k>j$, we have $Y_{k+1}\subseteq Y_k\subseteq Y_{j+1}$. Since $Y_{j+1}$ projects to $X^-$ in $X$, we get that both $Y_{k+1}\subseteq Y_k$ also project to $X^-$ in $X$. Then, similarly as above, $X^-$ and $X^+$ will have the same projection in $[Y_{k+1},Y_k]$.

For the second part, suppose a subnormal series (\ref{EquD}) is projectively isomorphic to a subnormal series (\ref{EquE}). Then each subfactor $[X_{i+1},X_i]$ in the first series projects onto a subfactor $[Y_{j+1},Y_j]$ of the second series such that $[Y_{j+1},Y_j]$ projects onto $[X_{i+1},X_i]$. Now, by the first part of the lemma, $X_{i+1}$ and $X_i$ will project to the same subgroup in any other subfactor $[Y_{k+1},Y_k]$, $k\neq j$ in (\ref{EquE}). Since all subgroups in the series (\ref{EquE}) are distinct, this implies that $[Y_{j+1},Y_j]$ is the unique subfactor in (\ref{EquE}) onto which $[X_{i+1},X_i]$ projects. Next, we will show that for every subfactor $[Y_{j+1},Y_j]$ in (\ref{EquE}) there is a unique subfactor $[X_{i+1},X_i]$ in (\ref{EquE}) which projects onto $[Y_{j+1},Y_j]$. Project all subgroups from the series (\ref{EquD}) in the subfactor $[Y_{j+1},Y_j]$. By monotonicity of projection, we will get a decreasing sequence of subgroups in the interval $[Y_{j+1},Y_j]$. This sequence will start with $Y_j$ (projection of $1$) and end with $Y_{j+1}$ (projection of $0$). Recall that each $[X_{i+1},X_i]$ either projects onto $[Y_{j+1},Y_j]$ or else, $X_{i+1}$ and $X_i$ project to the same subgroup in $[Y_{j+1},Y_j]$. This implies that there must exist unique $[X_{i+1},X_i]$ which projects onto $[Y_{j+1},Y_j]$, which is what we wanted to show. Thus, the mapping a subfactor from (\ref{EquD}) to a subfactor in (\ref{EquE}) onto which it projects is a bijection from the set of subfactors $[X_{i+1},X_i]$ in (\ref{EquD}) to the set of subfactors $[Y_{j+1},Y_j]$ in (\ref{EquE}). This implies that the series (\ref{EquD}) and (\ref{EquE}) must have the same number of terms. 
\end{proof}

The final and the main result of \cite{Isb} can now be recovered in the following form, in our abstract context. It can be derived from the previous results in a similar way as this is done in \cite{Isb}.

\begin{theorem}[Refinement Theorem]\label{refinement} Given two subnormal series (\ref{EquD}) and (\ref{EquE}) in a group $G$, the projections of the $Y_j$ in all $[X_{i+1},X_i]$ and the projections of the $X_i$ in all $[Y_{j+1},Y_j]$ form projectively isomorphic subnormal series, provided all of these projections except possibly $0_G$ are conormal.
\end{theorem}

\begin{proof} Consider the projections $Y_j^i=Y_j[X_{i+1},X_i]$ of $Y_j$ in $[X_{i+1},X_i]$ and the projections $X_i^j=X_i[Y_{j+1},Y_j]$ of $X_i$ in $[Y_{j+1},Y_j]$. We get the series
\begin{equation}
\label{EquF}    
X_0=Y_0^0\supseteq\dots\supseteq Y_m^0=X_1=Y_0^1\supseteq\dots\supseteq Y_m^1=X_2=Y_0^2\supseteq\dots \supseteq Y_m^{n-1}=X_n
\end{equation}
in the first case, and the series
\begin{equation}
\label{EquG}Y_0=X_0^0\supseteq\dots\supseteq X_n^0=Y_1=X_0^1\supseteq\dots\supseteq X_n^1=Y_2=X_0^2\supseteq\dots \supseteq X_n^{m-1}=Y_m
\end{equation}
in the second case. Note that these two series have the same number of terms: there are $mn+1$ many terms in each case. If all terms in these series are conormal (except possibly $0$), then each $[Y_{j+1}^i,Y_{j}^i]$ and each $[X_{i+1}^j,X_i^j]$ is a subfactor, by Lemma~\ref{butterfly}. Thus, eliminating duplicate terms in (\ref{EquF}) and in (\ref{EquG}) results in subnormal series. Consider a subfactor $[Y_{j+1}^i,Y_{j}^i]$ from the first subnormal series. Then $Y_{j+1}^i\neq Y_{j}^i$. Since $$[Y_{j+1}^i,Y_{j}^i]=[Y_{j+1},Y_j][X_{i+1},X_i]\textrm{ and } [X_{i+1}^j,X_{i}^j]=[X_{i+1},X_i][Y_{j+1},Y_j],$$
by the Butterfly Lemma (Lemma~\ref{butterfly}), $[Y_{j+1}^i,Y_{j}^i]$ and $[X_{i+1}^j,X_{i}^j]$ project onto each other and hence, in particular, $X_{i+1}^j\neq X_{i}^j$ and so $[X_{i+1}^j,X_{i}^j]$ is a subfactor in the second subnormal series. So the two subnormal series are projectively isomorphic.  
\end{proof}

\section{Counterexample to Isbell's claim on coarsest refinements}\label{SecC}

The refinement theorem formulated in \cite{Isb} includes two additional statements to what we have included in the abstract version of the theorem (Theorem~\ref{refinement}): 
\begin{itemize}
\item[(E1)] Every subfactor $[X^-,X^+]$ contained in some $[X_{i+1},X_i]$ and onto which some subfactor smaller than $[Y_{j+1},Y_j]$ projects is contained in the projection $$[Y_{j+1},Y_j][X_{i+1},X_i].$$

\item[(E2)] Therefore the indicated refinements are the coarsest projectively isomorphic refinements.
\end{itemize}
These claims, however, are false even in the context of abelian groups. 
Consider the cyclic group $G=\mathbb{Z}_{6}$, whose elements are $0,1,2,3,4,5$. The subgroup lattice of $\mathbb{Z}_{6}$ is the diamond:
$$
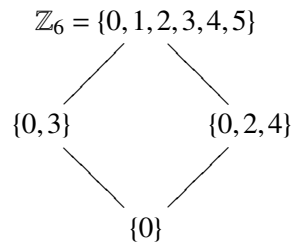
\xymatrix@!=15pt{ & \mathbb{Z}_{6}=\{0,1,2,3,4,5\} & \\ \{0,3\}\ar@{-}[ur] & & \{0,2,4\}\ar@{-}[ul] \\ & \{0\}\ar@{-}[ul]\ar@{-}[ur] & }$$
Consider the following subnormal series in this group:
$$X_0=\{0,1,2,3,4,5\}\supset X_1=\{0\},\quad Y_0=\{0,1,2,3,4,5\}\supset Y_1=\{0,2,4\}\supset Y_2=\{0\}.$$
The subfactor $[X^-,X^+]=[\{0,3\},\{0,1,2,3,4,5\}]$ is contained in $[X_1,X_0]$. The subfactor $[Y_2,Y_1]$ is contained in itself. However,  
\begin{align*}
[Y_2,Y_1][X^-,X^+]&=[(Y_2\wedge X^+)\vee X^-, (Y_1\wedge X^+)\vee X^-]\\
&=[\{0,3\},\{0,1,2,3,4,5\}]\\
&=[X^-,X^+]
\end{align*}
is not contained in the projection
$$[Y_2,Y_1][X_1,X_0]=[Y_2,Y_1],$$
unlike what is claimed in (E1). Now, to contradict (E2), we first note that the projectively isomorphic refinements produced by the two given subnormal series, using the construction described in Theorem~\ref{refinement}, are both the series $Y_0\supset Y_1\supset Y_2$. However, this pair of projectively isomorphic refinements is not the coarsest pair. The subnormal series
$$X_0\supset \{0,3\}\supset X_1$$
and $Y_0\supset Y_1\supset Y_2$
are projectively isomorphic. While they are refinements of the series $X_0\supset X_1$ and $Y_0\supset Y_1\supset Y_2$, respectively, the first of these is not coarser than $Y_0\supset Y_1\supset Y_2$.

The proof of (E1) and (E2) in \cite{Isb} is a single sentence claiming that the following implication holds:
$$[Y^{-},Y^+]\subseteq [Y_{j+1},Y_j]\Rightarrow [Y^{-},Y^+][X_{i+1},X_i]\subseteq [Y_{j+1},Y_j][X_{i+1},X_i].$$
While this claim holds indeed, it does not imply (E1), since $[Y^{-},Y^+][X_{i+1},X_i]$ need not be the same as $[X^{-},X^{+}]$ in (E1) (where $[Y^{-},Y^+]$ is the subfactor contained in $[Y_{j+1},Y_j]$ which projects onto $[X^{-},X^{+}]$). Indeed, in the counterexample above, $$[Y^{-},Y^+][X_{i+1},X_i]=[Y_2,Y_1][X_1,X_0]=[Y_2,Y_1]\neq [X^{-},X^{+}].$$

\section{On other abstract approaches to the butterfly lemma}\label{SecD}

The butterfly lemma and the refinement theorem have been abstracted to numerous general contexts. For instance, already the book on group theory by Zassenhaus \cite{zassenhaus} (originally published in 1958), discusses an abstraction in general lattices. Our context is more nuanced in that, firstly, we care about constructing the actual isomorphisms between groups instead of just establishing correspondences inside a subgroup lattice and secondly, because of the self-dual context, not all subgroups are necessarily conormal in our context. The first of these nuances requires an interplay between a lattice-theoretic context and a categorical context. The context of a $\gamma$-category of Burgin \cite{Bur}, for instance, provides such interplay. As remarked in \cite{BorGra}, this context brings together semi-abelian categories and Puppe-Mitchell exact categories \cite{Pup62,Mit65} (which are the same as pointed Grandis exact categories). It is evident from the properties of these categories discussed in \cite{BorGra} that subobjects in a $\gamma$-category constitute a noetherian form of the category if and only if they admit binary joins (a requirement fulfilled in the case of semi-abelian categories, but which is also expected to be fulfilled in the standard non-semi-abelian examples, such as the category of torsion-free abelian groups mentioned in \cite{BorGra}, for instance). In this noetherian form, all subgroups will be conormal. So up to the mild limitation of existence of joins of subobjects, the Jordan-H\"older Theorem obtained in \cite{BorGra} could be deduced from our refinement theorem as a consequence. It should be noted, however, that the approach in the proof of the Jordan-H\"older Theorem in \cite{BorGra} is different to a classical approach that passes via the refinement theorem.

Two different but similar approaches in abstracting the butterfly lemma are given in \cite{Wyl} and \cite{Tho}, which have inspired an abstraction in \cite{NS} to a fairly general framework based on the notion of a star-regular category \cite{gju12}. The contexts subsumed here include all regular categories \cite{BGvO} having certain pushouts (the case considered in \cite{Tho}). In this case, the role of subgroups and the subgroups normal to those are played by entities of different types, which makes it unclear how subnormal series could be defined and hence whether a refinement theorem can be established. It is of course nevertheless interesting that a form of the butterfly lemma can be established in cocomplete regular categories. The class of cocomplete regular categories is significantly wider than the class of semi-abelian categories. In fact, it includes all quasi-varieties of universal algebras. As shown in \cite{Jan14}, the formula $f^{-1}f(X)=X\vee \mathsf{Ker}f$ of Axiom~\ref{AxA} holds for the form of subobjects in a regular category if and only if it is protomodular \cite{Bou}. So the form of subobjects in a (cocomplete) regular category need not be noetherian. This does not exclude the possibility of there being another form over a (cocomplete) regular category that is noetherian.

The context of \cite{Wyl} is better comparable to that of a noetherian form. The refinement theorem has been obtained in this context under an additional mild axiom in \cite{FriWyl}. The axioms in \cite{Wyl,FriWyl} are fairly weak and hold, for instance, in ideal-determined categories \cite{JMTU}, which is a wider class of categories than the class of semi-abelian categories: an ideal determined category is semi-abelian if and only if it is protomodular. The form of subobjects of an ideal-determined category misses to be noetherian precisely because of the failure of the formula $f^{-1}f(X)=X\vee \mathsf{Ker}f$ of Axiom~\ref{AxA}. The reason why this formula is not needed for the approach to the butterfly lemma in \cite{Wyl} is that there, already in the formulation of the butterfly lemma, Wyler replaces the join $S\vee N$ of subgroups, when $N$ is a normal subgroup, with the construction $\pi_N^{-1}\pi_N(S)$.

While in some ways the context of a noetherian form is more restrictive than the categorical contexts discussed above, in other ways it is less restrictive: the contexts of \cite{Wyl,Bur,BorGra}, for instance, require pointedness of the category, while the contexts of \cite{Tho,NS} require existence of finite limits (and moreover, regularity of the category). We should also point out that the abstract approaches to the butterfly lemma and the refinement theorem discussed above do not establish these results in their full strength as formulated by Isbell \cite{Isb}, which was the main goal of the present paper.

\end{document}